\documentclass[letterpaper,10pt]{amsart}
\usepackage{amsmath,amssymb,amsxtra,amsthm, amstext, amscd,amsfonts,fancyhdr, hyperref,enumerate}
\usepackage[OT2,T1]{fontenc}
\DeclareSymbolFont{cyrletters}{OT2}{wncyr}{m}{n}
\DeclareMathSymbol{\Sha}{\mathalpha}{cyrletters}{"58}

\newcommand{\bC}{{\mathbb{C}}}

\newcommand{\bP}{{\mathbb{P}}}
\newcommand{\bQ}{{\mathbb{Q}}}
\newcommand{\bR}{{\mathbb{R}}}

\newcommand{\bZ}{{\mathbb{Z}}}

\newcommand{\Ba}{{\mathbf{a}}}

\newcommand{\Bc}{{\mathbf{c}}}

  \newcommand{\A}{{\mathcal{A}}}
  \newcommand{\B}{{\mathcal{B}}}

  \newcommand{\G}{{\mathcal{G}}}

  \newcommand{\M}{{\mathcal{M}}}
  \newcommand{\N}{{\mathcal{N}}}
\renewcommand{\O}{{\mathcal{O}}}

  \newcommand{\R}{{\mathcal{R}}}
\renewcommand{\S}{{\mathcal{S}}}

  \newcommand{\Z}{{\mathcal{Z}}}
  \newcommand{\HH}{\mathcal{H}}

\newcommand{\fp}{\mathfrak{p}}

\newcommand{\Gal}{\operatorname{Gal}}

\newcommand{\GL}{\operatorname{GL}}

\newcommand{\Aut}{\operatorname{Aut}}

\newcommand{\ep}{\varepsilon}

\newcommand{\ol}{\overline}

\newcommand{\upchi}{{\raise.35ex\hbox{$\chi$}}}

\newcommand{\BS}{\text{BS}}

\makeatletter
\@namedef{subjclassname@2010}{%
  \textup{2010} Mathematics Subject Classification}
\makeatother

\newtheorem{theorem}{Theorem}[section]

\newtheorem{proposition}[theorem]{Proposition}
\newtheorem{lemma}[theorem]{Lemma}

\theoremstyle{definition}

\newtheorem{remark}[theorem]{Remark}

\numberwithin{equation}{section}

\begin{document}

\title{On monic abelian cubics}

\author{Stanley Yao Xiao}
\address{Department of Mathematics \\
University of Toronto \\
Bahen Centre \\
40 St. George Street, Room 6290 \\
Toronto, Ontario, Canada \\  M5S 2E4 }
\email{syxiao@math.toronto.edu}
\indent


\begin{abstract} In this paper we prove the assertion that the number of monic cubic polynomials $F(x) = x^3 + a_2 x^2 + a_1 x + a_0$ with integer coefficients and irreducible, Galois over $\bQ$ satisfying $\max\{|a_2|, |a_1|, |a_0|\} \leq X$ is bounded from above by $O(X (\log X)^2)$. We also count the number of abelian monic binary cubic forms with integer coefficients up to a natural equivalence relation ordered by the so-called Bhargava-Shankar height. Finally, we prove an assertion characterizing the splitting field of 2-torsion points of semi-stable abelian elliptic curves.
\end{abstract}

\maketitle

\section{Introduction}
\label{Intro}

In the 19th century D.~Hilbert established the so-called Hilbert irreducibility theorem. One version of it can be stated as follows: when ordering degree $n$ monic polynomials
\[f(x) = x^n + a_{1} x^{n-1} + \cdots + a_n, a_i \in \bZ \text{ for } i = 1, \cdots, n\]
with the box height 
\begin{equation} \label{box h} H(f) = \max\{|a_1|, \cdots, |a_n|\}\end{equation}
proportion tending to 100\% of such polynomials will be irreducible and have Galois group isomorphic to the symmetric group $S_n$. \\

Hilbert's original proof of his theorem is not quantitative in the sense that it does not give a way to quantify how many degree $n$-polynomials of bounded box height fail to have $S_n$ as their Galois group. For any transitive subgroup $G \leq S_n$ and positive number $X \geq 1$, we write
\begin{equation} \label{NG} \N_G^{(n)}(X) = \# \{f(x) = x^n + a_1 x^{n-1} + \cdots + a_n \in \bZ[x], H(f) \leq X, \Gal(f) \cong G\}.\end{equation}
Van der Waerden proved that 
\[\N_{S_n}^{(n)}(X) = (2X)^n + O_n \left(X^{n -  \frac{6}{(n-2) \log \log n}}\right)\]
for $n \geq 3$. He conjectured that one should be able to replace the error term by $O_n\left(X^{n-1}\right)$, which is best possible since the subset of monic polynomials where the constant coefficient vanishes, all of which are reducible, already gives this order of magnitude. \\

A more precise formulation of van der Waerden's question is to ask whether or not one can obtain a sharper error term once the obvious reducible polynomials are removed, and indeed, to ask for asymptotic estimates for $\N_G^{(n)}(X)$ when $G \ne S_n$. \\

The simplest case of this question corresponds to $n = 3$ and $G = C_3$. Such polynomials are called \emph{abelian cubics}. It is well-known that an irreducible cubic polynomial with integer coefficients is abelian if and only if its discriminant is a square integer. \\

In this paper we give an estimate for $\N_{C_3}^{(3)}(X)$. We will prove the following:

\begin{theorem} \label{mt} Let $C_3$ be the cyclic group of order $3$ and $\N_{C_3}^{(3)}(X)$ given as in (\ref{NG}). Then there exist positive numbers $k_1, k_2$ such that for all $X > k_2$ we have
\begin{equation} 2X  \leq \N_{C_3}^{(3)}(X) < k_1 X (\log X)^2.
\end{equation}
\end{theorem}

One should compare Theorem \ref{mt} to the results regarding monic quartic polynomials obtained by Chow and Dietmann \cite{ChowDiet}. They proved that $\N_{G}^{(4)}(X) = o\left(X^{3 - \delta_G} \right)$ for all transitive proper subgroups $G$ of $S_4$, where $\delta_G$ is a positive number which depends on $G$. Most notably they obtained the exact asymptotic order of magnitude (but not an asymptotic formula) for $\N_{D_4}^{(4)}(X)$, namely that
\[\N_{D_4}^{(4)}(X) \asymp X (\log X)^2.\]
They also proved that $\N_{C_3}^{(3)}(X) = O_\ep \left(X^{3/2 + \ep}\right)$ for any $\ep > 0$. \\

Our Theorem \ref{mt} and Chow and Dietmann's theorem are the only results we are aware of that establishes the exact exponent when counting monic polynomials of degree $n \geq 3$ having Galois isomorphic to $G$ a proper subgroup of $S_n$ with respect to box height. \\

The upper bound in Theorem \ref{mt} should be considered the main contribution of this paper. The lower bound is given by a simple and classical construction (see for example \cite{Stew}). In view of the upper bound one should ask whether the lower bound or the upper bound is closer to the truth. We note that if we count monic \emph{totally reducible} cubic polynomials instead then we achieve the upper bound exactly. Indeed, such polynomials are characterized by triples of integers $r_1, r_2, r_3$ by
\[f(x) = (x - r_1)(x - r_2)(x - r_3) = x - (r_1 + r_2 + r_3)x^2 + (r_1 r_2 + r_1 r_3 + r_2 r_3)x - r_1 r_2 r_3.\]
It is clear that there are $O(X)$ such polynomials with at least two of $r_1, r_2, r_3 = 0$ and box height at most $X$, and $O(X \log X)$ such polynomials if exactly one of $r_1, r_2, r_3$ is zero. If $r_1, r_2, r_3 \ne 0$ then the condition $|r_1 r_2 r_3| \leq X$ implies that $|r_1 + r_2 + r_3|, |r_1 r_2 + r_1 r_3 + r_2 r_3| \ll X$, so there are $O(X (\log X)^2)$ such polynomials. Moreover, it is easy to choose $\gg X (\log X)^2$ triples $(r_1, r_2, r_3)$ such that $f(x)$ has height $H(f) \leq X$. If one considers abelian cubics to be comparable to totally reducible cubics, then the upper bound in Theorem \ref{mt} can be seen as best possible, and quite possibly the exact order of magnitude. \\

In order to prove Theorem \ref{mt} we first need to parametrize monic abelian cubic polynomials. Note that the set of monic cubic polynomials is invariant under translations. The action which sends $x \mapsto x + u$ has two basic polynomial invariants, which we denote by $I$ and $J$, given by
\begin{equation} \label{IJ def} I(F) = a_2^2 - 3a_1, J(F) = -2a_2^3 + 9a_2 a_1 - 27 a_0 \end{equation}
where $F(x) = x^3 + a_2 x^2 + a_1 x + a_0$. It follows that $F$ has a unique representation as 
\begin{equation} \label{trace 0} F\left(x - \frac{a_2}{3}\right) = x^3 - \frac{I(F)}{3} x - \frac{J(F)}{27}.\end{equation}

One can interpret this as an action of a subgroup of $\GL_2(\bZ)$ on the lattice of integral binary cubic forms. For the set of monic binary cubic forms, the natural action given above is realized by the upper triangular subgroup of $\GL_2(\bZ)$, namely
\[U(\bZ) = \left\{ \begin{pmatrix} 1 & n \\ 0 & 1 \end{pmatrix} : n \in \bZ \right\}.\]
The quantities $I(F), J(F)$ given in (\ref{IJ def}) are then invariants with respect to this action. In fact, all polynomial invariants of this action are generated by $I,J$. \\

To prove Theorem \ref{mt}, it is convenient to consider binary cubic forms rather than cubic polynomials. We thus need to parametrize monic binary cubic forms. It is well-known that for any monic cubic form that
\begin{equation} \label{disc IJ} \frac{4I(F)^3 - J(F)^2}{27} = \Delta(F).
\end{equation}
Since an irreducible cubic form $F$ is abelian if and only if $\Delta(F)$ is a square, it follows that we are required to study integer solutions to the equation
\[4z^3 = x^2 + 3y^2.\]
If $\gcd(x,y,z) = 1$, then the parametrization is provided in full by Cohen \cite{Coh}. However, it is not always the case that $\gcd(x,y,z) = 1$. We will show in Section \ref{stan form} that it suffices to study the equation
\begin{equation} \label{main disc eq} cx^3 = u^2 - uv + v^2, \gcd(x,u) = 1,
\end{equation}
where $c = c_1^2 - c_1 c_2 + c_2^2$ for $c_1, c_2 \in \bZ$; see Proposition \ref{para thm}. 

Of course, given the symmetry of (\ref{main disc eq}), the roles of $u,v$ may be swapped in Proposition \ref{para thm}. \\

Using Proposition \ref{para thm} we obtain the following parametrization of monic abelian cubics given by the shape (\ref{trace 0}):

\begin{theorem} \label{IJ can forms} Let $F(x,y) = x^3 + a_2 x^2 y + a_1 xy^2 + a_0 y^3 \in \bZ[x,y]$ be an irreducible cubic form such that $\Gal(F) \cong A_3$. Then $(I(F), J(F))$ is given by one of the following three possibilities:
\begin{equation} \label{tra 0} \begin{pmatrix} I(F) \\ \\ J(F) \end{pmatrix} = \begin{pmatrix} 9c(s^2 - st + t^2) \\  \\ 27c((2 c_1 - c_2)s^3 - 3(c_1 + c_2)s^2 t + 3(2c_2 - c_1) st^2 + (2 c_1 - c_2)t^3 ) \end{pmatrix},  \gcd(s,t) = 1
\end{equation}
where $c = c_1^2 - c_1 c_2 + c_2^2$ and $3 \nmid s^2 - st + t^2$,
\begin{equation} \label{tra 1} \begin{pmatrix} I(F) \\ \\ J(F) \end{pmatrix} = \begin{pmatrix} 3c(s^2 - st + t^2) \\ \\ 27c \left(c_2 s^3 + (c_1 - 3c_2)s^2 t - c_1 st^2 + c_2 t^3 \right)
\end{pmatrix}, \gcd(s,t) = 1, 
\end{equation}
with $c = c_1^2 - 3c_1 c_2 + 9c_2^2, 3 \nmid c_1, s^2 - st + t^2$, and
\begin{equation} \label{tr 2} \begin{pmatrix} I(F) \\ \\ J(F) \end{pmatrix} = \begin{pmatrix} c(s^2 - st + t^2) \\ \\ c\left((2c_1 - 3c_2)s^3 - 3(c_1 + 3c_2)s^2 t + 3(6c_2 - c_1)st^2 + (2c_1 - 3c_2)t^3\right) \end{pmatrix},  \gcd(s,t) = 1
\end{equation}
with $c = c_1^2 - 3c_1 c_2 + 9c_2^2, 3 \nmid c_1, s^2 - st + t^2$. 
\end{theorem} 

It turns out that a rather convenient way to establish Theorem \ref{IJ can forms} from Proposition \ref{para thm} is to first parametrize binary cubic forms (not necessarily monic) by their \emph{Hessian covariants}, or in the parlance of \cite{BhaShn}, by their \emph{shape}; see Proposition \ref{Prop1}. \\

A consequence of Proposition \ref{Prop1} is that we are able to recover a classical theorem in composition laws of rings and ideals of low rank, namely the correspondence between $3$-torsion of class groups of quadratic fields and cubic fields which are nowhere totally ramified (see \cite{HCL1} for a modern view of this phenomenon through composition laws). Our proof will perhaps highlight the phenomenon that the shape (Hessian covariant) of a cubic ring is able to identify certain arithmetic properties. In essence, we replace an explicit algebraic characterization of the map between nowhere totally ramified cubic rings and certain ideal classes of the corresponding quadratic field by identifying a cubic ring with certain integers representable by its Hessian covariant. \\

The $I,J$-invariants can be used to define a height for monic binary cubic forms, which is perhaps more natural than the box height. In \cite{BhaSha}, Bhargava and Shankar used analogous invariants to define a height on the space of binary quartic forms, which descends to a height on the space of monic binary cubic forms. We shall denote this height by the \emph{Bhargava-Shankar} height, given by
\begin{equation} \label{BS height} H_{\BS}(F) = \max\{I(F)|^3, J(F)^2/4\}.
\end{equation}
Observe that $H_{\BS}$ is only well-defined for monic binary cubic forms. \\

When restricted to abelian cubics, and the observation that $\Delta(F) = (4I(F)^3 - J(F)^2)/27$, it follows that $H_{\BS}(F) = I(F)^3$ for all $F$ abelian (since necessarily $\Delta(F) > 0$ in this case). We then have the following theorem:

\begin{theorem} \label{BST} Let $\M_{\BS}(X)$ denote the number of $U(\bZ)$-equivalence classes of irreducible binary cubic forms with integer coefficients and Galois over $\bQ$ with Bhargava-Shankar height bounded by $X$. Then 
\[\M_{\BS}(X) =  \frac{3}{2} X^{1/3} \log X + O\left(X^{1/3}\right). \]
\end{theorem}

One should compare Theorem \ref{BST} with the following statement enumerating monic binary cubic forms which are totally reducible over $\bQ$. Let $\M_{\BS}^\dagger(X)$ denote the number of $U(\bZ)$-equivalence classes of totally reducible binary cubic forms with integer coefficients, ordered by Bhargava-Shankar height. Then we have:
\begin{equation} \M_{\BS}^\dagger(X) =  c_0 X^{1/3} + O \left(X^{1/6}\right)
\end{equation}
for some positive number $c_0$. We remark that there is a result of G.~Yu \cite{Yu} which suggests that there ought to be $O_\ep\left(X^{1/3 + \ep}\right)$ $\GL_2(\bZ)$-equivalence classes of quartic forms with Galois group $V_4$ with Bhargava-Shankar height up to $X$. The consequence of Theorem \ref{BST} suggests that the same should be expected for $A_4$-quartic forms. \\

Another curiosity about elliptic curves that arises from Theorem \ref{IJ can forms} is the following. It is well-known that all elliptic curves $E/\bQ$ have a unique minimal Weierstrass model of the shape
\begin{equation} \label{Weierstrass} E: y^2 = x^3 - \frac{I}{3}x - \frac{J}{27},
\end{equation}
where $(I,J)$ has to satisfy some congruence condition modulo $27$. In view of  (\ref{trace 0}), it follows that an \emph{abelian} elliptic curve, or an elliptic curve where the corresponding cubic polynomial is abelian, has $(I,J)$ given by Theorem \ref{IJ can forms}. Note that an elliptic curve $E$ can be \emph{semi-stable} only if for all primes $p$, the corresponding cubic polynomial does not totally ramify. This implies that $\gcd(I,J) = 1$. This leads to the following conclusion:

\begin{theorem} \label{abelian EC} Let $E/\bQ$ be a semi-stable elliptic curve given by the Weierstrass model (\ref{Weierstrass}). Suppose that the attached cubic polynomial $f(x) = x^3 - Ix/3 - J/27$ is an abelian cubic polynomial. Then the splitting field of $f$ is $\bQ\left(\zeta_9 + \zeta_9^{-1}\right)$, where $\zeta_9$ is a primitive 9th root of unity. 
\end{theorem}

The outline of this paper is as follows. We first prove Theorem \ref{para thm}, which is necessary to parametrize our abelian cubic forms. Next, in Section \ref{stan form} we use the Hessian covariant of binary cubic forms to parametrize monic binary cubic forms. Then, using Proposition \ref{para thm}, we obtain a parametrization of monic abelian cubic forms. Section \ref{box count} contains the proof of Theorem \ref{mt}. Finally, auxiliary algebraic consequences, namely Theorems \ref{3-tors} and \ref{abelian EC}, are contained in Section \ref{3 abel}. 

\section{Parametrizing points on a family of genus 0 curves}

In this section, we solve (\ref{main disc eq}) in the following sense:
\begin{proposition} \label{para thm} The integer solutions to (\ref{main disc eq}) are parametrized by
\[x(s,t) = s^2 - st + t^2, u(s,t) = c_1 s^3 - 3 c_2 s^2 t + 3 (c_2 - c_1)st^2 + c_1 t^3 ,\]
\[v(s,t) = c_2 s^3 + 3(c_1 - c_2) s^2 t - 3c_1 st^2 + c_2 t^3, s, t \in \bZ, \gcd(s,t) = 1\]
and ranging over all pairs $c_1, c_2 \in \bZ$ such that $c = c_1^2 - c_1 c_2 + c_2^2$. 
\end{proposition}

Proposition \ref{para thm} is a simple consequence of the fact that the ring of Eisenstein integers is a unique factorization domain. \\

We treat (\ref{main disc eq}) as an equation over the Eisenstein integers $\bZ[\zeta_3]$, where $\zeta_3 = (-1 + \sqrt{-3})/2$. We further factor (\ref{main disc eq}) as 
\[(u + \zeta_3 v)(u + \zeta_3^2 v) = (c_1 + \zeta_3 c_2)(c_1 + \zeta_3^2 c_2)(x_1 + \zeta_3 x_2)^3 (x_1 + \zeta_3^2 x_2)^3,\]
where $c = c_1^2 - c_1 c_2 + c_2^2, a = s^2 - st + t^2$. The co-primality of $x$ with $u,v$ implies that $(x_1 + \zeta_3 x_2)^3$ must divide one of $u + \zeta_3 v, u + \zeta_3^2 v$. Without loss of generality, we assume that $(x_1 + \zeta_3 x_2)^3 | u + \zeta_3 v$ over $\bZ[\zeta_3]$. This implies that
\[u + \zeta_3 v = \zeta_3^k(c_1 + \zeta_3 c_2)(x_1 + \zeta_3 x_2)^3, k \in \{0,1,2\}. \]
We can absorb the $\zeta_3^k$ term into $c_1 + \zeta_3 c_2$, so it suffices to fix a value of $k$. For $k = 0$ we get the parametrization
\begin{align*} u  + \zeta_3 v & =  (c_1 + \zeta_3 c_2)(x_1^3 + 3 \zeta_3 x_1^2 x_2 + 3 \zeta_3^2 x_1 x_2^2 + x_2^3) \\
& = (c_1 + \zeta_3 c_2)(x_1^3 + x_2^3 - 3x_1 x_2^2 + 3\zeta_3 x_1 x_2(x_1 -x_2) ) \\
& = (c_1(x_1^3 + x_2^3 - 3x_1 x_2^2) -3 c_2 x_1 x_2 (x_1 - x_2) + \zeta_3(c_2(x_1^3 + x_2^3 - 3x_1^2 x_2 ) + 3c_1x_1 x_2 (x_1 - x_2))).
\end{align*}
Comparing coefficients we get that
\[v = c_2 x_1^3 + 3 (c_1 - c_2) x_1^2 x_2 - 3c_1 x_1 x_2^2 + c_2 x_2^3 \]
and
\[u  = c_1 x_1^3 - 3c_2 x_1^2 x_2 + 3(c_2 - c_1)x_1 x_2^2 + c_1 x_2^3.\]
The co-primality of $x$ and $u,v$ implies that $\gcd(x_1, x_2) = 1$. This completes the proof of the theorem. 

\section{Standard form of monic binary cubic forms and Hessian covariants}
\label{stan form} 

For a given binary cubic form
\[F(x,y) = a_3 x^3 + a_2 x^2y + a_1 xy^2 + a_0 y^3,\]
define the \emph{Hessian covariant} of $F$ to be 
\[H_F(x,y) = \frac{1}{4} \det \begin{pmatrix} \dfrac{\partial^2 F}{\partial x^2} & \dfrac{\partial^2 F}{\partial x \partial y} \\ \\\dfrac{\partial^2 F}{\partial x \partial y} & \dfrac{\partial^2 F}{\partial y^2} \end{pmatrix}.\]
Explicitly, we have
\begin{equation} \label{Hess} H_F(x,y) = (a_2^2 - 3 a_3 a_1)x^2 + (a_2 a_1 - 9 a_3 a_0)xy + (a_1^2 - 3 a_2 a_0)y^2 = Ax^2 + Bxy + Cy^2.
\end{equation}
For a binary quadratic form $g(x,y) = ax^2 + bxy + cy^2$, we define
\[V_g(\bC) = \{F(x,y) = a_3 x^3 + a_2 x^2 y + a_1 xy^2 + a_0 y^3 : H_F(x,y) \text{ is proportional to } g(x,y)\}.\]
We have the following 

\begin{proposition} \label{Prop1} Let $g(x,y) = ax^2 + bxy + cy^2 \in \bC[x,y]$ be a non-singular binary quadratic form with $a \ne 0$. Then
\begin{equation} \label{VgC def} V_g(\bC) = \left\{a_3 x^3 + a_2 x^2 y + \frac{ba_2 - 3ca_3}{a} xy^2 + \frac{(b^2 - ac)a_2 - 3bc a_3}{3a^2} y^3 : a_3, a_2 \in \bC  \right\}.\end{equation}   
\end{proposition}
This fact appears well-known; see for example \cite{BhaShn}. Nevertheless we give a proof of it for completeness. 

\begin{proof} We recall notation from \cite{X0}, where we dealt with the so-called \emph{Hooley matrix}:
\[\HH_F = \frac{1}{2 \Delta(H_F)} \begin{pmatrix} B \sqrt{-3\Delta(H_F)} - \Delta(H_F) & 2 C \sqrt{-3\Delta(H_F)} \\ -2A \sqrt{-3\Delta(H_F)} & -B \sqrt{-3\Delta(H_F)} - \Delta(H_F) \end{pmatrix},\]
where $A,B,C$ are as in (\ref{Hess}). It was shown by Hooley in \cite{Hoo} that $\HH_F$ is a stabilizer of $F$ with respect to the substitution action of $\GL_2$. Moreover, it was shown in \cite{X0} that for a given binary quadratic form $g(x,y) = ax^2 + bxy + cy^2$ with real coefficients and non-zero discriminant and associated matrix
\[\HH_g = \frac{1}{2 \Delta(g)} \begin{pmatrix} b \sqrt{-3\Delta(g)} - \Delta(g) & 2 c \sqrt{-3\Delta(g)} \\ -2a \sqrt{-3\Delta(g)} & -b \sqrt{-3\Delta(g)} - \Delta(g) \end{pmatrix},\]
that $\HH_g \in \Aut_\bR (F)$ if and only if $g$ is proportional to $H_F$. Here $\Aut_\bR(F)$ refers to the stabilizer subgroup of $F$ in $\GL_2(\bR)$ corresponding to the substitution action. Using this, one checks through explicit calculation that $H_F$ is proportional to $g$ if and only if $F$ is given as in (\ref{VgC def}). Similarly, that any element $F \in V_g(\bC)$ does indeed have Hessian covariant proportional to $g$ is easily checked. 
\end{proof}

\begin{remark} One can also prove Proposition \ref{Prop1} by observing that every binary cubic form $F$ with non-zero discriminant is $\GL_2(\bC)$-equivalent to $xy(x+y)$. 
\end{remark}

\begin{remark} Bhargava and Shnidman gave a slightly different form of the set $V_g(\bC)$ given in (\ref{VgC def}). Indeed, $V_g(\bC)$ corresponds to cubic forms of a given \emph{shape} $g$. 
\end{remark}

We now focus on monic binary cubic forms. Let $g(x,y)$ be the primitive integral binary quadratic form proportional to $H_F$. Since $\Delta(H_F) = -3 \Delta(F)$, it follows that $3 | \Delta(g)$ whenever $\Delta(F)$ is a square. Since $H_F$ is a covariant of $F$, it follows that $g$ is also a covariant of $F$. Applying the transformation in the lemma to $g$ shows that $9g(x+vy, y) \in \bZ[x,y]$. \\

Without loss of generality, we first translate $F$ by an integer, which enables us to assume that $a_2 \in \{-1,0,1\}$. We then further translate so that $F$ is of the form (\ref{trace 0}). Let $g(x,y) = ax^2 + bxy + cy^2$ be a primitive, integral binary quadratic form such that $H_F$ is proportional to $g$. It then follows from Proposition \ref{Prop1} that
\begin{equation} \label{trace 0 2} F(x,y) = x^3 - \frac{3c}{a}xy^2 - \frac{bc}{a^2} y^3.\end{equation}
Comparing (\ref{trace 0}) and (\ref{trace 0 2}), we see that if $F \in \bZ[x,y]$ then $I(F), J(F) \in \bZ$, and either $a_2 \equiv 0 \pmod{3}$ so $I(F) \equiv 0 \pmod{3}, J(F) \equiv 0 \pmod{27}$ or 
\[I(F) \equiv 1 \pmod{3}, J(F) \equiv (\pm 1)(2 - 9a_1) \pmod{27}.\]
Observe that $9a_1 \equiv 0, 9, 18 \pmod{27}$, so $(\pm 1)(9 a_1 - 2) \equiv \pm 2, \pm 7, \pm 16 \pmod{27}$. Next we see that 
\[\frac{9c}{a}, \frac{27bc}{a^2} \in \bZ.\]
Put $a = 3^k \alpha$, with $\gcd(\alpha, 3) = 1$. We then see that $\alpha | c$. Since $g$ is assumed to be primitive, it follows that $\gcd(\alpha, b) = 1$. It thus follows that $\alpha^2 | c$. As observed earlier, we have that $3 | \Delta(g)$. Thus if $k \geq 1$, then $3 | b$ and hence $3 \nmid c$. It follows that $k \leq 2$. \\

We first treat the case when $k = 0$. Then we have $c = a^2 c^\prime$ for $c^\prime \in \bZ$. It then follows that
\begin{equation} \label{tr 0} F(x,y) = x^3 - 3a c xy^2 - b c y^3, \gcd(a,b) = 1.
\end{equation}
We then see that $9c^\prime = \gcd(I(F), J(F))$. If $k = 1$ then we rewrite $(a,b,c)$ as $(3\alpha, 3 \beta, \alpha^2 \gamma)$, with $3 \nmid \alpha$. Then we see
\[ F(x,y) = x^3 - \alpha \gamma xy^2 - \frac{\beta \gamma}{3} y^3, \gcd(\alpha,\beta) = 1.
\]
In this case we have $J(F) = 9\beta \gamma$, so in fact $J(F) \equiv 0 \pmod{27}$. Since $3 \nmid \gamma$ it follows that $3 | \beta$, whence $9 | b$. We thus obtain the shape
\begin{equation} \label{tr 1} F(x,y) = x^3 - ac xy^2 - bc y^3, \gcd(a,b) = 1.
\end{equation} 
Finally, if $k = 2$ then we obtain
\begin{equation} \label{tr 2} F(x,y) = x^3 - \frac{ac}{3} xy^2 - \frac{bc}{27}y^3, \gcd(a,b) = 1.
\end{equation}
Comparing (\ref{tr 0}), (\ref{tr 1}), and (\ref{tr 2}) with (\ref{trace 0}) gives
\begin{equation} \label{IJ abc} (I,J) = \begin{cases} (9ac, 27bc), 3 \nmid a, \gcd(a,b) = 1 \\ 
 (3ac, 9bc), 3 \nmid a, \gcd(a,b) = 1 \\ 
 (ac, bc), 3 \nmid a, \gcd(a,b) = 1.
 \end{cases}
\end{equation}
Then it is easy to see that $\Delta(F)$ for $F$ given as in (\ref{tr 0}), (\ref{tr 1}), (\ref{tr 2}) is given by:

\begin{equation} \label{three fam} \Delta(F) = \begin{cases} 27c^2(4ca^3 - b^2) & \text{ if } F \text{ is given by } (\ref{tr 0}) \\ \\
 c^2(4ca^3 - 27b^2) & \text{ if } F \text{ is given by } (\ref{tr 1}) \\ \\
 \dfrac{c^2(4ca^3 - b^2)}{27} & \text{ if } F \text{ is given by } (\ref{tr 2}). 
\end{cases}
\end{equation}

Using (\ref{three fam}), we can write out all abelian cubic forms of the shape (\ref{trace 0}) which is a translate of an integral form. In each of the cases in (\ref{three fam}) we have an equation of the form
\begin{equation} \label{c first}  4ca^3 = u^2 + 3v^2, u,v \in \bZ, \gcd(a,u) = \gcd(a,v) = 1.\end{equation}
Note that the right hand side is a norm in $\bZ[\zeta_3]$, hence the left hand side must be as well. If $c$ is a norm in $\bZ[\zeta_3]$ then we are done. Otherwise, there must exist a prime $p$ which is not a norm in $\bZ[\zeta_3]$ and which divides $c$ with odd multiplicity. It follows that $p$ also divides $a$, and $p | u,v$. But then $a$ is not co-prime to $u,v$, contradicting our assumption. Thus $c$ must be a norm in $\bZ[\zeta_3]$. \\

We aim to obtain a parametrized set of solutions for this equation, following \cite{Coh}. First we massage (\ref{c first}). Note that $u^2 + 3v^2 = (u - v)^2 - (u - v)(-2v) + (2v)^2$. Put $u_1 = u - v, v_1 = -2v$, so that (\ref{c first}) becomes $4ca^3 = u_1^2 - u_1 v_1 + v_1^2$. Observe that by definition $v_1$ is even, thus the right hand side can be even if and only if $u_1$ is even. Now put $u_1 = 2x$ and $y = -v$, to obtain $ca^3 = x^2 - xy + y^2$, which is equivalent to (\ref{main disc eq}). We may then proceed with the proof of Theorem \ref{IJ can forms}.  

\subsection{Proof of Theorem \ref{IJ can forms}} We unwrap (\ref{three fam}) and Proposition\ref{para thm} to obtain the desired parametrization. In the first case, we have
\begin{equation} \label{abcn syz} 4ca^3 = b^2 + 3n^2, n \in \bZ.\end{equation}
We then see that
\begin{align*} b(s,t) & = 2\left(c_1 s^3 - 3 c_2 s^2 t + 3 (c_2 - c_1)st^2 + c_1 t^3 \right) - c_2 s^3 - 3(c_1 - c_2) s^2 t + 3c_1 st^2 - c_2 t^3 \\
& = (2 c_1 - c_2)s^3 - 3(c_1 + c_2)s^2 t + 3(2c_2 - c_1) st^2 + (2 c_1 - c_2)t^3 
\end{align*}
for some $c_1, c_2$ such that $c = c_1^2 - c_1 c_2 + c_2^2$. In the second case, we have
\[4ca^3 = 3b^2 + n^2, n \in \bZ,\]
whence 
\[ b(s,t) = c_2 s^3 + 3(c_1 - c_2)s^2 t - 3c_1 st^2 + c_2 t^3, c = c_1^2 - c_1 c_2 + c_2^2.\]
However, in this case more needs to be said. Since $I(F) \equiv 0 \pmod{3}$, it follows that $J(F) \equiv 0 \pmod{27}$. But this implies that $bc \equiv 0 \pmod{3}$. We had already deduced that $3 \nmid c$ in this case, so we must have $b \equiv 0 \pmod{3}$. Note that $b(s,t) \equiv c_2(s^3 + t^3) \pmod{3}$. If $c_2 \not \equiv 0 \pmod{3}$, then $s^3 + t^3 \equiv 0 \pmod{3}$. This implies that $n \equiv 0 \pmod{3}$, and since $3 \nmid c$, that $a \equiv 0 \pmod{3}$. This violates the fact that $\gcd(a,b) = 1$, whence $c_2 \equiv 0 \pmod{3}$. \\

Finally, suppose that the third case in (\ref{three fam}) occurs. Then once again we have 
\[a(s,t) = s^2 - st + t^2, b(s,t) = (2 c_1 - c_2)s^3 - 3(c_1 + c_2)s^2 t + 3(2c_2 - c_1) st^2 + (2 c_1 - c_2)t^3, s, t \in \bZ, \gcd(s,t) = 1.\]
But now we need to impose an additional congruence relation, on
\[n(s,t) = c_2 s^3 + 3(c_1 - c_2)s^2 t - 3c_1 st^2 + c_2 t^3.\]
Indeed we must have $n(s,t) \equiv 0 \pmod{3}$. For the same reasons as in the previous case, we conclude that $c_2 \equiv 0 \pmod{3}$.

\section{Counting monic abelian cubics by box height}
\label{box count}

In this section we count monic, abelian cubics by the naive box height. While the arguments given in this section are elementary, it is worthwhile to give a short description of the strategy to be carried out. \\

By Theorem \ref{IJ can forms}, each monic abelian cubic form can be put into a standard form with vanishing $x^2y$-coefficient. Our strategy is to first count monic forms of this shape, and then see which ones admits at least one translation $x \mapsto x + u/3$ with the property that the translated form has box height bounded by $X$. \\

It turns out when we fix $\gcd(I,J)$, the above condition turns into a question of counting integer solutions to 
\[N(x_1, x_2, x_3) \leq X,\]
where $N$ is a cubic decomposable form. This immediately shows that for a fixed value of $c = \gcd(I,J)$ the corresponding number of cubics is $O_c(X (\log X)^2)$. \\

To deal with varying (and possibly very large) values of $c$, we consider ranges of $a,c$ separately; that is, we restrict $a,c$ to distinct dyadic ranges, say
\[T_1 < a \leq 2T_1, T_2 < c \leq 2 T_2.\]
Our box height condition implies that $ac \ll X^2$, so naturally $T_1 T_2 \ll X^2$. We then devise arguments to deal with each relevant range of $T_1, T_2$. \\

We consider the first case of (\ref{three fam}); the other two cases being similar. We look for the number of $u \in \bQ$ with $3u \in \bZ$ such that the translated polynomial
\[F_u(x) = (x - u)^3 - 3ac(x-u) - bc = x^3 - 3ux^2 + 3(u^2 - ac)x - (u^3 - 3acu + bc)\]
satisfies $H(F_u) \leq X$. Suppose that $c = c_1^2 - c_1 c_2 + c_2^2$. By Theorem \ref{IJ can forms}, we have
\begin{equation} \label{abn para} a = s^2 - st + t^2, b = (2c_1 - c_2)s^3 - 3(c_1 + c_2)s^2 t + 3 (2c_2 - c_1)st^2 + (2c_1 - c_2)t^3, \end{equation}
\[n = - c_2 s^3 - 3(c_1 - c_2)s^2 t + 3c_1 st^2 - c_2 t^3. \]
It follows that the constant coefficient of $F_u(x)$ is given by
\begin{equation} \label{Gust} \G_{c_1,c_2}(u,s,t) = u^3 + 3c(s^2 - st + t^2)u + c((2c_1 - c_2)s^3 - 3(c_1 + c_2)s^2 t + 3 (2c_2 - c_1)st^2 + (2c_1 - c_2)t^3).\end{equation}
One then checks that 
\begin{align} \label{norm id} H_{\G_{c_1,c_2}}(u,s,t) & = \notag
\frac{1}{3} \begin{vmatrix} \dfrac{\partial^2 \G}{\partial u^2} & \dfrac{\partial^2 \G}{\partial u \partial s} & \dfrac{\partial^2 \G}{\partial u \partial t} \\ \\ \dfrac{\partial^2 \G}{\partial u \partial s} & \dfrac{\partial^2 \G}{\partial s^2} & \dfrac{\partial^2 \G}{\partial s \partial t} \\ \\ \dfrac{\partial^2 \G}{\partial u \partial t} & \dfrac{\partial^2 \G}{\partial s \partial t} & \dfrac{\partial^2 \G}{\partial t^2} \end{vmatrix}\\ \notag \\ 
& = -54c\G_{c_1,c_2}(u,s,t).
\end{align}

In other words, $H_{\G_{c_1,c_2}}$ is proportional to $\G_{c_1,c_2}$. The following lemma implies that $\G_{c_1,c_2}$ is a \emph{decomposable form}; that is, it splits into a product of three linear forms over $\bC$.

\begin{lemma} Let $G(x_1, x_2, x_3) \in \bC[x_1, x_2, x_3]$ be a ternary cubic form which does not have a square linear factor. Then $G$ is the product of three linear forms if and only if $G$ is proportional to its Hessian $H_G$. 
\end{lemma}

\begin{proof} Recall that the the intersection points of the cubic curve $C_G$ in $\bP^2(\bC)$ defined by $G = 0$ with the curve defined by $H_G = 0$ are exactly the inflection points of $C_G$. If $G$ is proportional to $H_G$, then these two curves are identical, so every point of $C_G$ is an inflection point. This implies that every component of $C_G$ is a line. The converse follows easily by explicit calculation. 
\end{proof}

 In fact it is easily checked that $\G_{c_1,c_2}(u,s,t)$ must necessarily split over $\bR$. 

\subsection{Proof of the upper bound in Theorem \ref{mt}} 
We consider dyadic ranges for $a,c$. In particular, we suppose that
\begin{equation} \label{abc bounds} T_1 < c \leq 2T_1, T_2 < a \leq 2T_2,\end{equation}
satisfying $T_1 T_2 \ll X^2$. We note the fact that there must exist $u \in \bZ$ with $3u \in \bZ$ such that 
\begin{equation} \label{u bd} 3|u^2 - ac|, |u^3 - 3acu - bc| \leq X.\end{equation}
Put $N(T_1, T_2)$ for the number of quintuples $(u, c_1, c_2, s, t)$ which satisfies (\ref{u bd}). \\

We view the expression 
\begin{equation} \label{fu} f(u) = u^3 - 3acu + bc \end{equation}
as a polynomial in $u$. By assumption, it has positive discriminant. We then use the cubic equation for cubic polynomials with positive discriminant.

\begin{lemma}[Cubic formula for cubic polynomials with three real roots] \label{cubic eq} Let $f(x) = x^3 - 3px + q$ be a real polynomial with three distinct real roots, so that $p > 0$. Then the roots $r_1, r_2, r_3$ of $f$ are given by
\[r_1 = 2p^{1/2} \cos \left(\frac{\theta}{3}\right), r_2 = 2p^{1/2} \cos \left(\frac{\theta + 2 \pi}{3}\right), r_3 = 2p^{1/2} \cos \left(\frac{\theta + 4 \pi}{3}\right),\]
where
\[\theta = \arccos \left(\frac{q}{2p^{3/2}}\right).\]
\end{lemma}

We write the form $\G_{c_1,c_2}(u,s,t)$ given by (\ref{Gust}) as
\begin{equation} \label{norm exp} \G_{c_1, c_2}(u,s,t) = (u - \xi_1 s - \xi_2 t)(u - \xi_2 s + (\xi_1 + \xi_2) t)(u + (\xi_1 + \xi_2)s - \xi_1 t ) = L_1 L_2 L_3.
\end{equation}
say, with $\xi_1, \xi_2 \in \ol{\bQ} \cap \bR$. Note that
\[\xi_1 \xi_2 (\xi_1 + \xi_2) = c(2c_1 - c_2) \asymp T_1^{3/2},\]
hence $\xi_1, \xi_2 \ll T_1^{1/2}$. \\

We proceed to show that very small values of $T_1, T_2$ do not cause any issues:

\begin{lemma} \label{2/3 lem} Suppose that $T_1 T_2 \ll X^{2/3}$. Then $N(T_1, T_2) \ll X$.
\end{lemma}

\begin{proof} The proof follows easily from the observation that the number of possible choices for $u$ is $O(X^{1/3})$. 
\end{proof}

For the sequel, we shall assume that $T_1 T_2 \gg X^{2/3}$. \\

We consider $u$ in a dyadic interval $(Y/2, Y]$ for some $Y \ll X$. We will see that when $Y$ is appreciably larger or smaller than $\sqrt{T_1 T_2}$ then the contribution to $N(T_1, T_2)$ will be negligible. Indeed, if $Y$ is much smaller or larger than $\sqrt{T_1 T_2}$ then 
\[L_i(u,s,t) \gg \max \left\{Y, \sqrt{T_1 T_2} \right\} \gg \sqrt{T_1 T_2}\]
for $i = 1,2,3$. Hence 
\[(T_1 T_2)^{3/2} \ll |\G_{c_1, c_2}(u,s,t)| \leq X,\]
which implies that $T_1 T_2 \ll X^{2/3}$ and we are done by Lemma \ref{2/3 lem}. We may thus assume that $Y \asymp \sqrt{T_1 T_2}$. \\

For a given quintuple $(u, c_1, c_2, s, t)$ we order the linear factors $L_1, L_2, L_3$ by
\[|L_1| \leq |L_2| \leq |L_3|.\]
If $|L_1| \gg \sqrt{T_1 T_2}$, then we see that $T_1 T_2 \ll X^{2/3}$ and again we are done by Lemma \ref{2/3 lem}. Hence we may assume that $|L_1| = o(\sqrt{T_1 T_2})$. Note that we must have $|L_3| \gg \sqrt{T_1 T_2}$. These observations imply that 
\[\left \lvert \frac{\partial}{\partial u} \G_{c_1, c_2}(u,s,t) \right \rvert = \left \lvert L_1 L_2 + L_1 L_3 + L_2 L_3 \right \rvert \gg \lvert L_2 L_3 \rvert,\]
whence 
\[|L_2| \ll \frac{X}{\sqrt{T_1 T_2}}. \]
Put
\[L_i(u,s,t) = u - \ell_i(s,t), i = 1,2,3.\]
The binary cubic form $n_{c_1, c_2}(s,t)$ is precisely given by 
\[c n_{c_1, c_2}(s,t) = (\ell_1(s,t) - \ell_2(s,t))(\ell_1(s,t) - \ell_3(s,t))(\ell_2(s,t) - \ell_3(s,t)).\]
Since $\ell_i - \ell_j = L_i - L_j$ for $1 \leq i < j \leq 3$, it follows that $|\ell_1 - \ell_2| = |L_1 - L_2| \ll X (T_1 T_2)^{-1/2}$ and $|\ell_1 - \ell_3|, |\ell_2 - \ell_3| \asymp \sqrt{T_1 T_2}$. Hence
\begin{equation} \label{cn bd} c \cdot n_{c_1, c_2}(s,t) \ll X \sqrt{T_1 T_2}.\end{equation}

Next let us put 
\[f_{\pm X}(u) = u^3 - 3acu - bc \pm X\]
and let $r_i^\pm$ be the corresponding roots of $f_{\pm X}$. The possible solutions $u$ to (\ref{u bd}) given $c_1, c_2, s,t$ then lie in the three intervals 
\[[r_1^{-}, r_1^{+}], [r_2^{-}, r_2^+], [r_3^-, r_3^+].\]
Note that these intervals need not be disjoint. Typically we expect that these intervals are very short: the only exception is when 
\[\theta = \arccos \left(\frac{bc}{2(ac)^{3/2}} \right)   \]
given as in Lemma \ref{cubic eq} is very close to zero. We quantify this by writing
\[\frac{bc}{2(ac)^{3/2}} = 1 - \frac{\eta}{2(ac)^{3/2}}.\]
This implies that 
\begin{align*} \cos \theta & = 1 - \frac{\theta^2}{2} + O(\theta^4) \\
& = 1 - \frac{\eta}{2(ac)^{3/2}},
\end{align*}
which shows that
\[\theta^2 + O(\theta^4) = \frac{\eta}{(ac)^{3/2}}.\]
This bound is trivial if $\eta (ac)^{-3/2} \gg 1$ but measures how close $\theta$ is to zero when $\eta = o((ac)^{3/2})$. \\

We now note that, by (\ref{cn bd}), $\theta$ will be very close to zero if $T_1 T_2 \gg X$.

\begin{lemma} \label{small theta} Suppose that $T_1 T_2 \gg X$ and let $f(x)$ be given as in (\ref{fu}). Then $\theta$, defined as in Lemma \ref{cubic eq}, satisfies $\theta = O \left(X/T_1 T_2 \right)$. 
\end{lemma} 

\begin{proof} We consider $n = n_{c_1, c_2}(s,t)$ in the equation
\begin{equation} \label{disc eq} 4ca^3 = b^2 + 3n^2\end{equation} 
and factoring over $\bZ[\sqrt{-3}]$, we write the right hand side as 
\[(b + n \sqrt{-3})(b - n \sqrt{-3}).\]
Viewing the first vector as a complex number and expressing it in polar coordinates we see that
\[b + n \sqrt{-3} = 2(ca^3)^{1/2} e^{i \theta}\]
with $\theta$ as in the statement of the Lemma. Further, we have
\[n  = \frac{2(ca^3)^{1/2}}{\sqrt{3}} \sin(\theta). \]
By our bounds on $n$ given in (\ref{cn bd}) we see that
\begin{equation} \label{angle bd} \sin(\theta) = O \left(\frac{X \sqrt{T_1 T_2}}{(T_1 T_2)^{3/2}} \right) = O \left(\frac{X}{T_1 T_2}\right).\end{equation} 
By looking at the Taylor expansion of $\sin(\theta)$ around $0$ we conclude that $\theta = O\left(X (T_1 T_2)^{-1} \right)$, as desired. 
\end{proof} 

In particular, Lemma \ref{small theta} implies that whenever $T_1 T_2 \gg X$ we have $\eta/(ac)^{3/2} \ll 1$. \\

We shall first assume that $\eta/(ac)^{3/2} \gg 1$, and so $T_1 T_2 \ll X$. In this case we see that $\theta$ is bounded
away from zero. We expand the series of 
\[\cos\left(\frac{\theta}{3}\right), \sin\left(\frac{\theta}{3} \right)\]
and note that 
\[\cos(\alpha + 2\pi/3) = \frac{-1}{2} \cos \alpha - \frac{\sqrt{3}}{2} \sin \alpha, \cos(\alpha + 4\pi/3) = \frac{-1}{2} \cos \alpha + \frac{\sqrt{3}}{2} \sin \alpha.\]
If we write 
\[\theta_{\sharp} = \arccos \left(\frac{bc + X}{2(ac)^{3/2}} \right) \text{ and } \theta_{\flat} = \arccos \left(\frac{bc - X}{2(ac)^{3/2}} \right), \] 
we see that 
\begin{equation} \label{cos dif} \cos\left(\frac{\theta_{\sharp}}{3}\right) - \cos\left(\frac{\theta_{\flat}}{3} \right) = O \left(\lvert \theta_\sharp - \theta_\flat \rvert \right) = O \left(\frac{X}{(ac)^{3/2}}\right),\end{equation}
since $\max\{\theta_\sharp, \theta_\flat\} \gg 1$. Similarly, $|\sin(\theta_\sharp/3) - \sin(\theta_\flat/3)| = O(X/(ac)^{3/2})$ and by Lemma \ref{cubic eq} we conclude that $u$ must lie in a union of three intervals each having length $O \left(X (T_1 T_2)^{-1} \right)$. This immediately shows that
\begin{equation} \label{large gap sum}  \sum_{T_1 T_2 \ll X} N^\dagger(T_1, T_2) \ll \sum_{T_1 T_2 \ll X} T_1 T_2 = O(X \log X),\end{equation}
where the $\dagger$ indicates only those $(c_1, c_2, s, t)$ for which $\eta/(ac)^{3/2} \gg 1$ are counted. Here we used the trivial bound $O(T_1 T_2)$ to count such $(c_1, c_2, s,t)$ and for each such quadruple with $\eta/(ac)^{3/2} \gg 1$ there are $O(X (T_1 T_2)^{-1} + 1)$ possibilities for $u$. \\

We now turn our attention to the case when $\eta/(ac)^{3/2} \ll 1$ (but with no restriction on the size of $T_1 T_2$). In this case we note that (\ref{cos dif}) still holds, since 
\[\theta_\sharp^2 - \theta_\flat^2 = O \left(\frac{((\eta + X)^{1/2} - (\eta - X)^{1/2})(\eta^{1/2}) }{(T_1 T_2)^{3/2}} \right) = O \left(\frac{X}{(T_1 T_2)^{3/2}} \right). \]
Hence we see 
\begin{align*} \left \lvert \cos\left(\frac{\theta_\sharp + 2\pi}{3} \right) - \cos \left(\frac{\theta_\flat + 2 \pi}{3} \right) \right \rvert & = O \left(\frac{X}{(T_1 T_2)^{3/2}} \right) + \frac{|\theta_\sharp - \theta_\flat|}{2 \sqrt{3}} + O \left(|\theta_\sharp - \theta_\flat|^3 \right) \\
& \asymp \frac{X}{\eta^{1/2} (T_1 T_2)^{3/4}}.
\end{align*} 
If $\eta$ is much smaller than $X$ then we no longer have disjoint intervals, and the longest interval has length $O \left(X^{1/2}/(T_1 T_2)^{1/4} \right)$. We see then that the number of possible $u$'s is 
\begin{equation} \label{u pos} \begin{cases} O\left(\dfrac{X}{\eta^{1/2} (T_1 T_2)^{1/4}} + 1 \right) & \text{if } \eta \gg X \\ \\
O \left(\dfrac{X^{1/2}}{(T_1 T_2)^{1/4}} + 1 \right) & \text{if } \eta \ll X. \end{cases} \end{equation} 

Put $N(T_1, T_2, T_3)$ for the set of quintuples $(u, c_1, c_2, s, t)$ for which $a = s^2 -st + t^2, c = c_1^2 -  c_1 c_2 + c_2^2$ satisfies (\ref{disc eq}) and $T_3 < \eta \leq 2T_3$.\\

We factor $b + n \sqrt{-3}$ over $\bZ[\sqrt{-3}] \subset \bC$ into $\Bc \cdot \Ba^3$, say, where 
\[\Bc = c^{1/2} e^{i \gamma}\]
and
\[\Ba = a^{1/2} e^{i \alpha}.\]
We then have
\begin{equation} \label{angle sum condition} \gamma +  3 \alpha + 2 k \pi = \theta = O\left(\frac{T_3^{1/2}}{(T_1 T_2)^{3/4}} \right),\end{equation}
where $k = 0, 1, 2$. The situation is now essentially symmetric in $\Ba, \Bc$. Suppose, say, that $T_1 \ll T_2$ (so in particular $T_1 \ll X$). Then we first fix a vector $\Bc$, and then choose a vector having norm $a \in (T_2, 2T_2]$ lying in one of three sectors of angle $O \left(X (T_1 T_2)^{-1} \right)$. (\ref{angle sum condition}) gives three sectors depending on the value of $k$. Call one of these sectors $\A_\gamma$, say. If we have two vectors 
\[\Ba_1 = p_1 + \sqrt{-3} \cdot q_1, \Ba_2 = p_2 + \sqrt{-3} \cdot q_2 \in \A_\gamma,\]
with corresponding angles $\alpha_1, \alpha_2$ then 
\[\sqrt{3} \cdot |p_1 q_2 - p_2 q_1| = \lVert \Ba_1 \rVert \lVert \Ba_2 \rVert |\sin(\alpha_1 - \alpha_2)| = O \left( \frac{T_3^{1/2}}{(T_1 T_2)^{3/4}} \cdot T_2\right) = O \left(\frac{T_3^{1/2} T_2^{1/4}}{T_1^{3/4}} \right).\]
For each $\kappa = O \left(T_3^{1/2} T_2^{1/4} /T_1^{3/4}  \right)$ with $|p_1 q_2 - p_2 q_1| = \kappa$ there are at most $O(1)$ possibilities for $\Ba_2 \in \A_\gamma$ once $\Ba_1$ is fixed, since any different solution would be separated by $\lVert \Ba_1 \rVert \gg T_2^{1/2}$. Hence having fixed $\Bc$ we see that $\A_\theta$ contains $O \left(T_3^{1/2} T_2^{1/4} /T_1^{3/4} + 1 \right)$ possibilities for $\Ba$. Thus the number of choices for $\Ba, \Bc$ is 
\[O \left(T_3^{1/2} T_1^{1/4} T_2^{1/4} + T_1 \right).\]
If instead we have $T_2 \leq T_1$ then we switch tracks and fix $\Ba$ first. The argument proceeds in an identical manner except now there is only one sector $\B_\alpha$, say. Using symmetry in this way allows us to conclude that there are 
\begin{equation} \label{sym T1T2} O \left(T_3^{1/2} T_1^{1/4} T_2^{1/4} + \min \left\{T_1, T_2 \right\} \right)\end{equation}
possibilities for $\Ba, \Bc$. \\

By (\ref{u pos}) the number of choices for $u$ is then $\displaystyle O\left(\frac{X}{T_3^{1/2} (T_1 T_2)^{1/4}} + 1 \right)$ if $T_3 \gg X$. Thus we have
\[N(T_1, T_2, T_3) = O \left(X  + T_3^{1/2} T_1^{1/4} T_2^{1/4} + \frac{X^{1/2} \min\{T_1, T_2\}}{(T_1 T_2)^{1/4}} + \min\left\{T_1, T_2 \right\}\right).\]
Noting that $\theta \ll X (T_1 T_2)^{-1}$ by (\ref{angle bd}), we see that
\[T_3 \asymp \eta \ll \theta^2 (T_1 T_2)^{3/2} \ll \frac{X^2}{T_1^{1/2} T_2^{1/2}}.\]
This implies that 
\[T_3^{1/2} T_1^{1/4} T_2^{1/4} \ll \frac{X}{T_1^{1/4} T_2^{1/4}} \cdot (T_1 T_2)^{1/4} \ll X. \]
Further, we see that $(T_1 T_2)^{1/4} \gg \left(\min\{T_1, T_2\}\right)^{1/2}$, hence
\[\frac{X^{1/2} \min\{T_1, T_2\}}{(T_1 T_2)^{1/4}} \ll X^{1/2} \min\{T_1, T_2\}^{1/2}.\]
We thus obtain
\begin{align*} \sum_{X \ll T_3 \ll X^2/(T_1^{1/2} T_2^{1/2})} \sum_{X^{2/3} \ll T_1 T_2 \ll X^2} N(T_1, T_2, T_3) & \ll \sum_{T_3 \ll X^2} \sum_{\substack{T_1 \ll X \\ T_1 \ll T_2 \ll X^2/T_1}}  O \left(X + X^{1/2} T_1^{1/2} + T_1 \right) \\
& \ll X (\log X)^2.
\end{align*} 
If instead $T_3 \ll X$ then we use the second bound from (\ref{u pos}), which shows that
\begin{align*} N(T_1, T_2, T_3) & = O \left(X^{1/2} T_3^{1/2} + T_3^{1/2} T_1^{1/4} T_2^{1/4} + \frac{X^{1/2} \min \{T_1, T_2\}}{(T_1 T_2)^{1/4}} + \min\{T_1, T_2\}\right) \\
& = O\left(X^{1/2} T_3^{1/2} + T_3^{1/2} T_1^{1/4} T_2^{1/4} + X^{1/2} \min\{T_1, T_2\}^{1/2} + \min\{T_1, T_2\} \right).\end{align*} 
Summing over dyadic ranges we again obtain the bound $O(X (\log X)^2)$. This is sufficient for the proof of the upper bound of Theorem \ref{mt}.  

\subsection{Proof of the lower bound in Theorem \ref{mt}} For the lower bound, we have that the set
\[S(X) = \{x^3 + ax^2 + (a-3)x - 1 : a \in [-X,X] \cap \bZ\} \]
contains $2X + O(1)$ elements. Let $f_a$ denote the element in $S(X)$ corresponding to the parameter $a$. Then $\Delta(f_a) = (a^2 - 3a + 9)^2$, so $f_a$ is either an abelian cubic or is totally reducible over $\bQ$. The latter situation occurs only when $f_a$ has a rational integer root. But the constant coefficient of $f_a$ is $-1$, so this root must be $\pm 1$. We then check that 
\[f_a(1) = 2a - 3, f_a(-1) = 1\]
are both odd, so they cannot be zero. Hence $f_a$ is irreducible for all $a \in \bZ$ and thus $f_a$ is an abelian cubic for all $a \in \bZ$. This provides the required lower bound.

\section{Counting monic abelian cubics by invariants}

In this section we prove Theorem \ref{BST}. Since in all cases our parametrization demands that $\gcd(s,t) = 1$, we first address this issue. Put
\[\S(T) = \{(s,t) \in \bZ^2 : \gcd(s,t) = 1, s^2 - st + t^2 \leq T\}\]
and $S(T) = \# \S(T)$. Next, put $\Z(T) = \{(s,t) \in \bZ^2 : s^2 - st + t^2 \leq T\}$ and $Z(T) = \# \Z(T)$. Then for any positive number $M$ we have
\begin{equation} \S(T) = \prod_{\substack{p < M \\ p \ne 3}} \left(1 - \frac{1}{p^2}\right) Z(T) + O \left(\sum_{M < p \ll T^{1/2} } \frac{Z(T)}{p^2} \right).
\end{equation}
Note that the infinite product satisfies
\begin{align*} \prod_{p > M} \left(1 - \frac{1}{p^2}\right) & =  \exp\left(\sum_{p > M} \log\left(1 - \frac{1}{p^2}\right) \right) \\
& = \exp\left(\sum_{p > M} \left(- \frac{1}{p^2} - \frac{1}{2p^4} - \frac{1}{3p^6} - \cdots \right) \right) \\
& = \exp \left(- \frac{c_p}{p^2} \right),
\end{align*}
where $c_p = \sum_{n = 1}^\infty \frac{1}{n p^{2n - 2}}$ is an absolute constant. It follows that
\[\prod_{p > M} \left(1 - \frac{1}{p^2}\right) = 1 + O(p^{-2}).\]
From here one concludes that 
\begin{equation} S(T) = \frac{27}{4\pi^2} \cdot \frac{\pi T}{\sqrt{3}} + O\left(T^{1/2}\right).
\end{equation}

We proceed to treat the first case given by (\ref{tra 0}). Thus we are required to count the solutions $(c_1, c_2, s, t)$ satisfying the inequality
\begin{equation} (c_1^2 - c_1 c_2 + c_2^2)(s^2 - st + t^2) \leq X^{1/3}/9
\end{equation}
and $\gcd(s,t) = 1, 3 \nmid s^2 - st + t^2$. We then have
\begin{equation} \label{I eq} \sum_{c = c_1^2 - c_1 c_2 + c_2^2 \leq X^{1/3}} S(X^{1/3}/9c)  = \sum_{c_1^2 - c_1 c_2 + c_2^2 \leq X^{1/3}} \left(\frac{X^{1/3}}{\pi \sqrt{3} c} + O \left(\frac{X^{1/6}}{c^{1/2}}\right) \right).
\end{equation}
Let $r_3(n) = \# \{(s,t) \in \bZ^2 : n = s^2 - st + t^2\}$. Observe that 
\[\sum_{n \leq Y} r_3(n) = \# \{(s,t) \in \bZ^2 : s^2 - st + t^2 \leq Y\} = \frac{\pi Y}{\sqrt{3}} + O \left(Y^{1/2}\right).\]
We then have, for any $\alpha > 0$,
\begin{align*} \sum_{c_1^2 - c_1 c_2 + c_2^2 \leq Y} (c_1^2 - c_1 c_2 + c_2^2)^{-\alpha} & = \sum_{n \leq Y} \frac{r_3(n)}{n^\alpha} \\
& = Y^{-\alpha} \sum_{n \leq Y} r_3(n) + \int_1^{Y} \frac{\sum_{n \leq t} r_3(n)}{t^{\alpha + 1}} dt \\
& = \frac{\pi Y^{1 - \alpha}}{\sqrt{3}} + \int_1^{Y} \left( \frac{\pi}{\sqrt{3} t^{\alpha}} + O \left(\frac{1}{t^{\alpha + 1/2}} \right) \right) dt.
\end{align*}
The values we require to evaluate (\ref{I eq}) are $\alpha = 1$ and $\alpha = 1/2$, giving the term
\[\frac{ X^{1/3} \log X}{4}  + O \left(X^{1/3}\right).\]
However, we must remember to impose the condition that $s^2 - st + t^2 \not \equiv 0 \pmod{3}$, which introduces a factor of $2/3$ to the main term. Hence we obtain the asymptotic form
\begin{equation} \label{I eq 2} \frac{ X^{1/3} \log X}{6}  + O \left(X^{1/3}\right). 
\end{equation}
The cases corresponding to (\ref{tr 0}) and (\ref{tr 2}) correspond respectively to the inequalities
\begin{equation} \label{ineq inv 1} (c_1^2 - c_1 3c_2 + 9c_2^2)(s^2 - st + t^2) \leq X^{1/3}/3
\end{equation}
and
\begin{equation} \label{ineq inv 3} (c_1^2 - c_1 3c_2 + 9c_2^2)(s^2 - st + t^2) \leq X^{1/3}.
\end{equation}
However, there are now additional congruence relations that must be satisfied by $c_1, c_2, s,t$ as indicated in Theorem \ref{IJ can forms}. For the second case we must have $c_1 \not \equiv 0 \pmod{3}$ and $s^2 - st + t^2 \not \equiv 0 \pmod{3}$. These conditions introduce a factor of $4/9$. This gives that there are
\begin{equation} \label{I eq 1} \frac{X^{1/3} \log X}{3} + O \left(X^{1/3}\right)\end{equation}
possibilities in this case. Finally, in the third case we apply the same congruence restrictions, resulting in the estimate
\begin{equation} \label{I eq 3} X^{1/3} \log X + O \left(X^{1/3}\right).\end{equation}
Thus we see that 
\begin{align*} \M_{\BS}(X) & = \frac{1}{12} \left(2 + 4 + 12\right) X^{1/3} \log X + O \left(X^{1/3}\right) \\ 
& = \frac{3X^{1/3} \log X}{2} + O \left(X^{1/3}\right),
\end{align*}
as desired.

\section{Some algebraic consequences}
\label{3 abel} 

In this Section we record some nice algebraic consequences of the methods we develop in this paper which may be of independent interest. Firstly, we shall give another proof of the following well-known theorem in algebraic number theory: 

\begin{theorem} \label{3-tors} The $3$-torsion part of narrow class groups of quadratic fields are in one-to-one bijection with maximal, irreducible nowhere totally ramified cubic rings. 
\end{theorem} 

\subsection{Proof of Theorem \ref{3-tors}} Let $\R_2, \R_3$ denote respectively the $\GL_2(\bZ)$-equivalence classes of binary quadratic and cubic forms. We will show that the map $\phi_{3,2} : \R_3 \rightarrow \R_2$ sending a binary cubic form $F$ to its Hessian covariant $H_F$ induces a bijection between the two objects in the theorem. Indeed it is well-known that $\GL_2(\bZ)$-classes of binary cubic forms with square-free discriminant precisely correspond to rings of integers of cubic fields which are nowhere totally ramified, and $\GL_2(\bZ)$-classes of binary quadratic forms correspond to ideal classes of quadratic fields. \\

Let $F$ be a binary cubic form with integer coefficients. Since $\Delta(H_F) = -3 \Delta(F)$, it follows that $F$ has square-free discriminant only if $H_F$ is primitive. For a given binary quadratic form $g(x,y) = ax^2 + bxy + cy^2$ with co-prime integer coefficients and non-zero discriminant, we have that an element $F = F_{a_3, a_2} \in V_g(\bC)$ with integer coefficients given in Proposition \ref{Prop1} has discriminant equal to
\[\Delta(F_{a_3,a_2}) = \frac{g(a_2, -3a_3)^2(4ac - b^2 )}{3a^4}.\]
We now apply an element of $\GL_2(\bZ)$ to $g$ (respectively $F$) to replace $a$ with a prime $p$ representable by $g$ which does not divide $\Delta(g)$. The prime $p$ can be interpreted as representing the narrow class associated to $g$. Moreover, we see that $\Delta(F)$ can be square-free only if $g$ represents $p^2$ as well. \\

Since $g$ represents $p$, it follows that $p$ splits in $\bQ\left(\sqrt{\Delta(g)}\right)$. We thus factor $(p) = \fp_1 \fp_2$ and without loss of generality, we assume that the ideal class corresponding to $g$ is represented by $\fp_1$. Since $g$ also represents $p^2$, which has the possible factorizations 
\[(p^2) = (p)(p), \fp_1^2 \fp_2^2 , \fp_2^2 \fp_1^2,\]
it follows that $\fp_1^2$ or $\fp_2^2$ must be in the same class as $\fp_1$ since the first case corresponds to imprimitive representations. Indeed, examining the congruence conditions in (\ref{VgC def}) shows that the second case also cannot happen. Thus $\fp_1, \fp_2^2$ must lie in the same class. Note that the class $[\fp_2]$ of $\fp_2$ is the inverse of the class $\fp_1$, whence $[\fp_1]^3 = \text{Id}$. This shows that $g$ is an order 3-element in the ideal class group of $\bQ(\sqrt{\Delta(g)})$. This establishes the desired bijection.

\subsection{Proof of Theorem \ref{abelian EC}} In this subsection we give a proof of Theorem \ref{abelian EC}, which asserts that all semi-stable abelian elliptic curves have a common $2$-torsion field, equal to the maximal real subfield of $\bQ(\zeta_9)$. \\

The cubic polynomials we are considering take the shape
\begin{equation} \label{EC1} f(x) = x^3 - 3(s^2 - st + t^2)x \pm (s^3 - 6s^2 t + 3st^2 + t^3),
\end{equation}
by Proposition 9.6 in \cite{Coh}. By explicit calculation, we see that the Hessian covariant of $F(x,y) = y^3 f(x/y)$ is proportional to
\[g(x,y) = (s^2 - st + t^2)x^2 \pm (s^3 - 6s^2 t + 3st^2 + t^3)xy + (s^2 - st + t^2)^2 y^2.\]
One then checks that 
\[g(x,y) = u^2 + uv + v^2,\]
where
\[u = sx + (s^2 - t^2)y, v = -tx + (2st - s^2)y.\]
Moreover, for $G(x,y) = x^3 + 3x^2 y - y^3$, we have
\[G(u,v) = (s^3 - 3s^2 t + t^3)F(x,y).\]
This shows that $G$ and $F$ have the same splitting fields. Note that
\[\theta_0 = \arccos \left(\frac{-1}{2}\right) = \frac{\pi}{3}.\]
It then follows from Lemma \ref{cubic eq} that the roots $r_1, r_2, r_3$ of $G(x,1)$ are given by
\[r_1 = 2 \cos(\theta_0/3) = 2 \cos \left(\frac{2\pi}{9}\right), r_2 = 2 \cos \left(\frac{8\pi}{9}\right), r_3 = 2 \cos \left(\frac{14 \pi}{9}\right).\]
These are precisely equal to $\zeta_9 + \zeta_9^{-1}, \zeta_9^2 + \zeta_9^{-2}, \zeta_9^4 + \zeta_9^{-4}$, where $\zeta_9 = \exp\left(2\pi i/9\right)$ is a primitive 9th root of unity. This completes the proof.


\end{document}